\documentclass[12pt,a4paper]{amsart}
\usepackage{url}
\usepackage{amsmath,amsthm,amsfonts,amssymb,latexsym,yhmath}
\newtheorem{theorem}{Theorem}[section]

\newtheorem{lemma}[theorem]{Lemma}

\theoremstyle{definition}
\newtheorem{definition}[theorem]{Definition}

\newtheorem{question}[theorem]{Question}

\newcommand{\U}{\mathcal U}
\newcommand{\w}{\omega}

\newcommand{\IQ}{\mathbb Q}

\newcommand{\IP}{\mathbb P}

\newcommand{\B}{\mathcal{B}}

\newcommand{\V}{\mathcal{V}}

\newcommand{\vid}{\wideparen{\ \ }}
\newcommand{\bigvid}{\wideparen{\ \ }}

\newcommand{\Lev}{\mathit{Lev}}

\newcommand{\uhr}{\upharpoonright}

\newcommand{\name}[1]{\dot{#1}}
\newcommand{\la}{\langle}
\newcommand{\ra}{\rangle}

\newcommand{\Split}{\mathrm{Split}}

\newcommand{\forces}{\Vdash}

\newcommand{\W}{\mathcal W}

\newcommand{\nothing}[1]{}

\title[Preservation of  $\gamma$-spaces]{Preservation of  $\gamma$-spaces and covering properties of products }

\author[D. Repov\v{s}  and L. Zdomskyy]{Du\v{s}an Repov\v{s}  and Lyubomyr Zdomskyy}

\address{Faculty of Education, and Faculty of Mathematics and Physics,
University of Ljubljana, 1000 Ljubljana, Slovenia.}
\email{dusan.repovs@guest.arnes.si}
\urladdr{http://www.fmf.uni-lj.si/\~{}repovs/index.htm}

\address{Institut f\"ur Diskrete Mathematik und Geometrie, Technische Universit\"at Wien, Wiedner Hauptstra\ss e 8-10/104, 1040 Wien, Austria.}
\email{lzdomsky@gmail.com}
\urladdr{http://dmg.tuwien.ac.at/zdomskyy/}


\thanks{The first author was partially supported by
the Slovenian Research Agency grant P1-0292 and N1-0083.
The second author would
like to thank  the Austrian Science Fund FWF (Grants I 2374-N35 and I 3709-N35)
 for generous support for this research.}

\begin{document}
\begin{abstract}
We prove that  the Hurewicz property is not preserved by finite products in the Miller model.
This is a consequence of the fact that  Miller forcing preserves ground model $\gamma$-spaces.
\end{abstract}

\maketitle

\section{Introduction}

When trying to describe  $\sigma$-compactness in terms of open covers, Hurewicz
 \cite{Hur25} introduced
  the following property, nowadays called \emph{ the Menger property}: A topological space
$X$ is said to have this  property if for every sequence $\la \U_n : n\in\omega\ra$
of open covers of $X$ there exists a sequence $\la \V_n : n\in\omega \ra$ such that
each $\V_n$ is a finite subfamily of $\U_n$ and the collection $\{\cup \V_n:n\in\omega\}$
is a cover of $X$. The current name (the Menger property) has been adopted because Hurewicz
proved in   \cite{Hur25} that for metrizable spaces his property is equivalent to
a certain basis property considered by Menger in \cite{Men24}.
If in the definition above we additionally require that $\{\cup\V_n:n\in\w\}$
is a \emph{$\gamma$-cover} of $X$
(this means that the set $\{n\in\w:x\not\in\cup\V_n\}$ is finite for each $x\in X$),
then we obtain the definition of the Hurewicz covering property introduced
in \cite{Hur27}. Contrary to a conjecture of Hurewicz,
the class of  metrizable spaces having the Hurewicz property
turned out  to be  wider than the class of $\sigma$-compact spaces \cite[Theorem~5.1]{COC2}.

Like for most of the topological properties, it is interesting to ask whether the Hurewicz
property is preserved by finite products.
One of the motivations behind this question comes from spaces of continuous functions, see \cite[Theorem~21]{KocSch03}.
 In the case of general topological spaces
there are ZFC examples of Hurewicz spaces whose product is not even Menger, see  \cite[\S 3]{Tod95}
and the discussion in the introduction of \cite{RepZdo17}.
That is why \emph{we concentrate in what follows on subspaces of the Cantor space $2^\w$.}
 (Let us note that
the preservation of the Hurewicz property by finite products of
 metrizable spaces
reduces to  subspaces of $2^\w$, see the end of the proof   of \cite[Theorem~1.1]{RepZdo17} on p. 331 of that paper.)
The covering properties of products of subspaces of $2^\w$ with
 the Hurewicz property  turned out to be  sensitive to
the ambient set-theoretic universe: Under CH there exists a Hurewicz space
whose square is not Menger, see \cite[Theorem~2.12]{COC2}.
Later,  a similar construction has been carried out under a much weaker assumption, see
 \cite[Theorem~43]{TsaSch02}.
 In particular, under the Martin
Axiom there are Hurewicz subspaces of the Cantor space whose product is not Menger.

On the other hand, the product of any two Hurewicz subspaces of $2^\w$ is Menger in the Laver
and Miller models, see \cite{RepZdo17} and \cite{Zdo??}, respectively. In the Miller model we actually know that
 the product of  finitely many Hurewicz subspaces of $2^\w$
is Menger (for the Laver model this is unknown even for three Hurewicz subspaces),
because in this model the Menger property is preserved by products of subspaces of $2^\w$, see \cite{Zdo??}.
This is why the Miller model seemed to be the best candidate for a model where the Hurewicz
property is preserved by finite products of metrizable spaces. Our next theorem
refutes this expectation, and hence the question whether one can find ZFC examples of
Hurewicz subspaces $X,Y$ of $2^\w$ with non-Hurewicz product remains open.

Standardly, by the Miller model we mean a forcing extension of a ground model of GCH
by adding a generic for the forcing obtained by an iteration of length $\w_2$ with countable supports of
the poset defined by Miller in \cite{Mil84}. We recall the definition of this poset in the proof of Lemma~\ref{examples}.

\begin{theorem}  \label{main}
In the Miller model there are two $\gamma$-subspaces $X,Y$ of $2^\w$ such that $X\times Y$
is not Hurewicz. In particular, in this model the Hurewicz property is not preserved by finite products of
metrizable spaces.
\end{theorem}

A family $\U$ of subsets of a set $X$ is called  an \emph{$\w$-cover} of $X$ if $X\not\in\U$ and for every finite subset $F$ of
$X$ there exists $U\in\U$ such that $F\subset U$.
 A space $X$  is  called a \emph{$\gamma$-set} if every open $\w$-cover
 of $X$ contains a $\gamma$-subcover.
This notion was introduced  in
\cite{GerNag82} where it was proved that a Tychonoff space
 $X$ is a $\gamma$-space if and only if
the space $C_p(X)$ of all continuous functions from $X$ to $\mathbb R$ with
the topology of the pointwise convergence, has the Fr\'echet-Urysohn property, i.e.,
for every $f\in C_p(X)$ and $A\subset C_p(X)$ with $f\in\bar{A}$ there exists a sequence
$\la f_n:n\in\w\ra\in A^\w$ converging to $f$.

It is well-known that $\gamma$-spaces have the Hurewicz property in all finite powers,
see, e.g., \cite[Th.~3.6 and Fig. 2]{COC2} and references therein. This follows from the following characterization
proved in \cite{GerNag82}: $X$ is a $\gamma$-space if and only if for every sequence
$\la\U_n:n\in\w\ra$ of open $\w$-covers of $X$ there exists a sequence
$\la U_n\in\U_n:n\in\w\ra$ such that $\{U_n:n\in\w\}$ is a $\gamma$-cover of $X$.

Our proof of Theorem~\ref{main} is based on the fact that if $X\subset 2^\w$,  $X\in V,$
and $X$ is a $\gamma$-space in $V,$ then $X$ remains a $\gamma$-space in the forcing extension
by a countable support iteration of posets satisfying property $(\dagger)$ introduced in
Definition~\ref{def_dag} below.
This seems to be the first attempt to find iterable properties of forcing posets guaranteeing the preservation of
 ground model $\gamma$-spaces. Previously, only specific posets were treated:
By \cite{MilTsaZdo16} and \cite{Sch10} $\gamma$-spaces are preserved by Cohen and random forcing, respectively,
whereas the Hechler forcing kills all ground model uncountable $\gamma$-spaces, see \cite{Mil05}.

Let us note that  Cohen forcing satisfies $(\dagger)$
but fails to preserve Hurewicz spaces, see the discussion in \cite{MilTsaZdo16} after Problem 4.1 therein.
This is why our proof of Theorem~\ref{main} leaves open the following question:

\begin{question}
%Does  Miller forcing preserve the Hurewicz property of ground model spaces?
%What about subspaces of $2^\w$?
Does  Miller forcing preserve the Hurewicz property of ground model metrizable spaces containing no topological copies of $2^\w$?
What about Sierpi\'nski spaces?
\end{question}

\section{Proof of Theorem~\ref{main}}

Theorem~\ref{main} is a direct consequence of Lemmata~\ref{just_def},
\ref{dad_impl_gampres}, \ref{dag_preservation}, and \ref{examples} proved below,
combined with one of the main results of \cite{MilTsaZdo16}.
We shall consider only posets $\IP$ such that below any $p\in\IP$ there exist  incompatible $r,q$.
This is not an essential restriction
because most of the posets considered in literature have this property.
First we need to introduce some auxiliary notions.

\begin{definition} \label{def_dag}
\begin{itemize}
\item
A poset $\IP$ has property $(\dagger)$ if for every countable
elementary submodel $M\ni \IP$ of $H(\theta)$ for  big
enough $\theta$, every $p\in \IP\cap M$, and $\phi_i : \IP\cap M\to\IP\cap M$  for
all $i\in\w$ such that $\phi_i(p)\leq p$ for all $p\in\IP\cap M$ and
$i\in\w$, there exists an $(M,\IP)$-generic $q\leq p$ forcing
$$\name{G}\cap\{\phi_i(p):p\in M\cap\IP\} \mbox{ is infinite for all }i\in\w,$$
 where $\name{G}$ is the canonical $\IP$-name for the
$\IP$-generic filter.
\item
Let $\B=\{B_n:n\in\w\}$ be a bijective enumeration of  the standard
clopen base of the topology on $2^\w$, i.e., $\B$ consists of finite
unions of elements of the family $\big\{[s]=\{x\in 2^\w:x\uhr|s|=s\}:s\in 2^{<\w}\big\}$.
 Let $X\subset 2^\w$ and $M\ni X$ be as above.
$\W\subset\B$ is called $\la X,M,\w\ra$-hitting if $\W\cap\U$ is
infinite for every $\w$-cover $\U$ of $X$ such that $\U\in M$ and
$\U\subset\B$.
\item
The poset $\IP$ is called \emph{$\la X,\gamma\ra$-preserving} if for
every countable elementary submodel $M$ such that  $X,\IP\in M$,  $\la
X,M,\w\ra$-hitting $\W\subset \B$, and $p\in \IP\cap M$ there exists
an $(M,\IP)$-generic condition $q\leq p$ forcing $\W$ to be $\la X,
M[\name{G}],\w\ra$-hitting.
\end{itemize}
\end{definition}

In what follows we shall denote by $\Omega(X)$ and $\Gamma(X)$ the family of all
open $\w$- and $\gamma$-covers of a topological space $X$, respectively.
The following lemma justifies our terminology.

\begin{lemma} \label{just_def}
If $\IP$ is  $\la X,\gamma\ra$-preserving and $X\subset 2^\w$ is a
$\gamma$-set, then $X$ remains a $\gamma$-set in $V^{\IP}$.
\end{lemma}
\begin{proof}
Let $\name{\U}$ be a $\IP$-name for an $\w$-cover of $X$ by elements
of $\B$, $p\in\IP,$ and $M\ni\name{\U},p$ be a countable elementary
submodel. Let $\{\U_i:i\in\w\}$ be an enumeration of $\Omega(X)\cap
M\cap\mathcal P(\B)$ and $U_i\in\U_i$ be such that
$\W=\{U_i:i\in\w\}\in\Gamma(X)$. Then  $\W $ is $\la
X,M,\w\ra$-hitting,  and hence there exists an $(M,\IP)$-generic
$q\leq p$ forcing $\W\cap\name{\U}$ to be infinite. Thus $q$ forces
$\W\cap\name{\U}$ to be a $\gamma$-subcover of $\name{\U}$.
\end{proof}

\begin{lemma} \label{dad_impl_gampres}
If $\IP$ satisfies $(\dagger)$, then it is  $\la
X,\gamma\ra$-preserving for every $X\subset 2^\w$.
\end{lemma}
\begin{proof}
Let us enumerate $V^{\IP}\cap M$ as   $\{\name{\U}_i:i\in\w\}$. For
every $p\in\IP\cap M$ and $i\in\w$, if $p$
 does not force $\name{\U}_i$ to be an $\w$-cover of $X$ consisting of elements of $\B$, we can find $r_{i,p}\leq p$
which forces that $\name{\U}_i$ is not an $\w$-cover of $X$ by
elements of $\B$. Otherwise we set $\U_{i,p}=\{B\in\B:\exists r\leq p
(r\forces B\in\name{\U}_i)\}$ and note that
$\U_{i,p}\in\Omega(X)\cap M$. Furthermore,
 by the elementarity we have that for every $B\in\U_{i,p}$
there exists $M\ni r\leq p$ such that $r\forces B\in\name{\U}_i$.
 Let $\{p_n:n\in\w\}$ be an enumeration of
 $M\cap\IP$ and for every $n,i$ set $\U'_{i,p_n}=\U_{i,p_n}\setminus \{B_k : k\leq n\}$.
Since $\W$ is $\la X,M,\w\ra$-hitting, $|\W\cap\U'_{i,p}|=\w$ for
every $p\in M\cap \IP$ and $i\in\w$. For every $p,i$ as above pick
$U_{i,p}\in\W\cap\U'_{i,p}$ and $r_{i,p}\leq p$ such that
$r_{i,p}\in M$ and $r_{i,p}\forces U_{i,p}\in\name{\U}_i$.

Now let us fix $p_*\in \IP\cap M$ and  consider   maps
$\phi_i:p\mapsto r_{i,p}$, $i\in\w$. It follows that there exists an
$(M,\IP)$-generic $q\leq p_*$ forcing the set
$\name{G}\cap\{r_{i,p}:p\in\IP\cap M\}$ to be infinite for all
$i\in\w$. Let $G\ni q$ be $\IP$-generic and $i\in\w$. If
$\name{\U}_i^G$ is an $\w$-cover of $X$ by elements of $\B$, then no
$r_{i,p}\in G$ can force the negation thereof, and hence for each
such $r_{i,p}$ we have $U_{i,p}\in\W\cap\name{\U}_i^G$. Therefore
$|\W\cap \name{\U}_i^G|=\w$ since no $B\in\B$ can belong to
$\U'_{i,p}$ for infinitely many $p\in M\cap\IP$.
 \end{proof}

\noindent\textbf{Remark.}
It is a simple exercise to check that if in the definition of $(\dagger)$ we restrict ourselves
to only one $\phi:M\cap\IP\to M\cap\IP$ then we get an equivalent statement.
The longer formulation which we have chosen seems to be easier to apply, though.
\hfill $\Box$
\medskip

By the definition  we have that for every $X\subset 2^\w$ finite
iterations of  $\la X,\gamma\ra$-preserving posets are again $\la
X,\gamma\ra$-preserving. The proof of the next fact is modelled after that of
\cite[Lemma~2.8]{Abr10}. In fact, we just ``add an $\epsilon$'' to it, using ideas from  \cite{Dow90}.

\begin{lemma} \label{dag_preservation}
Let $X\subset 2^\w$. Then countable support iterations of $\la
X,\gamma\ra$-preserving posets are again $\la
X,\gamma\ra$-preserving.
\end{lemma}
\begin{proof}
We shall inductively prove the following formally stronger statement:
\begin{quote}
Let   $\la
\IP_\alpha,\name{\IQ}_\alpha:\alpha<\delta \ra$ be a  countable support iteration of $\la X,\gamma\ra$-preserving posets,
 $M$  a countable elementary submodel of $ H(\lambda)$ for a sufficiently large regular cardinal $\lambda$ such that $\delta,\IP_\delta\in M$, and $\W\subset\B$ be $\la X,M,\w\ra$-hitting.
For any $\delta_0\in\delta\cap M$ and $(M,\IP_{\delta_0})$-generic condition $q_0$  forcing $\W$ to be
$\la X, M[\name{G}_{\delta_0}],\w\ra$-hitting,  the following holds: If
$\name{p}_0\in V^{\IP_{\delta_0}}$ is such that
$$q_0\forces_{\IP_{\delta_0}}\name{p}_0\in\IP_\delta\cap M \mbox{ and } \name{p}_0\uhr\delta_0\in\name{G}_{\delta_0},$$
where $\name{G}_{\delta_0}$ is the canonical name for the $\IP_{\delta_0}$-generic, then there is an $(M,\IP_\delta)$-generic
condition $q$ such that
$$ q\uhr\delta_0=q_0 \mbox{ and } q\forces_{\IP_\delta}\:``\name{p}_0\in\name{G}_\delta \:\wedge\:\W\mbox{ is }\la X,M[\name{G}_\delta],\w\ra\mbox{-hitting.''} $$
\end{quote}
We are going to prove this statement by induction on $\delta$, the only non-trivial case
 (modulo \cite[Lemma~2.6]{Abr10} and the proof
thereof) is when $\delta$ is a limit ordinal.
Fix a  strictly increasing sequence $\la\delta_n:n\in\w\ra$
of ordinals in $M$ cofinal in $M\cap\delta$.
  For every
$\nu<\mu<\delta$ let us denote by $\IP_{[\nu,\mu)}$ a $\IP_{\nu}$-name for the
iteration of $\name{\IQ}_\beta$, $\beta\in\mu\setminus\nu$, in $V^{\IP_{\nu}}$.
As usual,
(see, e.g., \cite{Lav76}) we shall identify  $\IP_{[\nu,\mu)}$ with the set of  all functions $p$ with domain $\mu\setminus \nu$
such that $1_{\IP_\nu}\bigvid p\in\IP_\mu$, ordered as follows:
Given a $\IP_\nu$-generic $G$ and $p_0,p_1\in\IP_{[\nu,\mu)}$,
$p_1^G\leq p_0^G$ in $\IP_{[\nu,\mu)}^G$ if there exists an $s\in G$
such that $s\bigvid p_1\leq s\bigvid p_0$ in $\IP_\mu$.

Set $D_0=\IP_\delta$ and let
$\{D_i:i\geq 1\}$ be the set of all open dense subsets of $\IP_\delta$
which belong to $M$ and $\{\name{\U}_i:i\geq 1\}$  an enumeration
of $V^{\IP_\delta}\cap M$ such that each $\tau\in V^{\IP_\delta}\cap M$
equals $\name{\U}_i$ for infinitely many $i$.    We shall define by
induction on $n\in\w$ a condition $q_n\in\IP_{\delta_n}$ and a name
$\name{p}_n\in V^{\IP_{\delta_n}}$ such that:
\begin{itemize}
\item[$(1)$] $q_0$ and $\name{p}_0$ are like in the quoted claim at the beginning of the proof; $q_n$ is
$(M,\IP_{\delta_n})$-generic; $q_{n+1}\uhr \delta_n=q_n$;
\item[$(2)$]  $\name{p}_n$ is a
$\IP_{\delta_n}$-name such that\\
 \centerline{$q_n\forces_{\IP_{\delta_n}}\  ``\name{p_n} \mbox{ is  a condition in } \IP_\delta\cap M \mbox{ such that }$}
\begin{itemize}
\item[$(a)$] $\name{p}_n\uhr \delta_n\in\name{G}_{\delta_n}$;
\item[$(b)$] $\name{p}_n\leq\name{p}_{n-1}$;
\item[$(c)$] $\name{p}_n\in D_n$; and
\item[$(d)$] If $n\geq 1$ then  $\name{p}_{n}$ decides whether $\name{\U}_n$ is an $\w$-cover of $X$ by elements of
$\B$, and in the case when decided to be such a cover, $\name{p}_n$ forces, in addition, that
$\exists m\geq n
(B_m\in\name{\U_n}\cap\W)$.''
\end{itemize}
\end{itemize}
 Assume that $q_n$ and $\name{p}_n$ have already been constructed.
For a while we shall work in $V[G]$, where  $G\ni q_n$ is $\IP_{\delta_n}$-generic. Then $p_n:=\name{p}_n^{G}\in D_n\cap M$
and $p_n\uhr\delta_n\in G$.
Find $p'_n\leq p_n$ such that $p'_n\uhr \delta_n\in G$ and $p'_n\in D_{n+1}\cap M$.
It exists because the set
$$ D'=\{p'\in\IP_{\delta_n}: (p'\perp p_n\uhr\delta_n)\vee(\exists p_n'\in D_{n+1} (p_n'\leq p_n \wedge p'=p_n'\uhr\delta_n))\} $$
is dense in $\IP_{\delta_n}$ and belongs to $M$, and hence $D'\cap M$ is predense below $q_n$, which yields $D'\cap G\cap M\neq\emptyset$.
Moreover, since $p_n\uhr\delta_n\in G$, any $p'\in D'\cap G$ is compatible with $p_n\uhr\delta_n$.
It follows that for any $p'\in G\cap D'\cap M$, any $p_n'\in M$ witnessing for $p'\in D'$ is as required.

Without loss of generality we may assume that each condition in $D_{n+1}$ decides whether $\name{\U}_{n+1}$ is an $\w$-cover of
$X$ by elements of $\B$. If $p'_n$ decides that it is not, then we set $p_{n+1}=p'_n$
and take $q_{n+1}$ to be any $(M,\IP_{\delta_{n+1}})$-generic  satisfying $(1), (2)$ and forcing over $\IP_{\delta_{n+1}}$
that $\mathcal W$ is $\la X,M[\name{G}_{\delta_{n+1}}],\w\ra$-hitting, its existence following by our inductive assumption.
Otherwise fix a $\IP_{\delta_{n}}$-name $\name{p}'_n\in M$ for a condition in $\IP_\delta$ such that $q_n$ forces that $\name{p}'_n$
has all the properties of $p'_n$ stated above, and
an $(M,\IP_{\delta_{i+1}})$-generic  $q_{n+1}$ such that $q_{n+1}\uhr\delta_n=q_n$,
 $q_{n+1}\forces_{\IP_{\delta_{n+1}}} \name{p}'_n\uhr\delta_{n+1}\in\name{G}_{\delta_{n+1}}$,
and $q_{n+1}\forces_{\IP_{\delta_{n+1}}}`` \W$ is $\la X, M[\name{G}_{\delta_{n+1}}],\w\ra$-hitting''.
Consider the $\IP_{\delta_{n+1}}$-name $\name{\W}_{n+1}$ which equals
\begin{eqnarray*}
 \big\{ \la r,\check{B}\ra\: :\: B\in\mathcal B \ \ \&\ \  \IP_{\delta_{n+1}}\ni r \mbox{ decides } \name{p}_n' \mbox{ as } p_n'\ \  \& \\
 \mbox{ exists } p\leq p_n' \mbox{ such that } p\uhr\delta_{n+1}=r \mbox{ and } p\forces_{\IP_\delta}\check{B}\in\name{\U}_{n+1}     \big\}.
\end{eqnarray*}
It follows that $\name{\W}_{n+1}\in M$ is a $\IP_{\delta_{n+1}}$-name which is forced by $q_{n+1}$ to be  an $\w$-cover of $X$ by elements of $\B$,
and hence  $q_{n+1}\forces_{\IP_{\delta_{n+1}}} |\W\cap\name{\W}_{n+1}|=\w$.
Let $H\ni q_{n+1}$ be $\IP_{\delta_{n+1}}$-generic over $V$ and $p'_n$  the interpretation $(\name{p}'_n)^H$.
Now we shall work in $V[H]$ for a while. It follows from the above that there exists $m>n$ such that
$B_m\in\W\cap\name{\W}_{n+1}^H$. Consequently, there exist
$r\in H$ and $p\leq p_n'$ such that $p\uhr\delta_{n+1}=r$ and $p\forces_{\IP_\delta}\check{B}_m\in\name{\U}_{n+1}$.
By  elementarity we can find such $r$ in $M$ (note that $M[H]\cap \IP_{\delta_{n+1}}= M\cap \IP_{\delta_{n+1}}$),
and hence  we can also find $p\in M$ as above.
Now  let $\name{p}_{n+1}\in M$  be a $\IP_{\delta_{n+1}}$-name
such that $q_{n+1}$ forces that $\name{p}_{n+1}$ has all the properties of $p$  stated above.
Its existence follows by the maximality principle.
This completes our inductive construction.

 Exactly as in the proof of
\cite[Lemma 2.8]{Abr10} one can verify that $q=\bigcup_{n\in\w}q_n$ is $(M,\IP_\delta)$-generic.
More precisely, it is easy to see by induction on $n$ that $q$ forces over $\IP_\delta$
that $\name{p}_{n+1}\leq \name{p}_n\in \name{G}_{\delta}\cap M$ for all $n\in\w$. Using this we are going to prove
 that each $D_n\cap M$ is predense below $q$. Suppose not. Then we can  find
$q'\leq q$ which is incompatible with all elements of $D_n\cap M$ for some $n\in\w$. Let $H\ni q'$ be $\IP_\delta$-generic.
Then $p_n:=\name{p}_n^H\in H\cap M\cap D_n$ by $(2)$, and hence $p_n$ is a condition in $D_n\cap M$
compatible with $q$ (because $q\in H$), a contradiction.

It suffices to note that
 $(2)(d)$ clearly ensures that
$q$ forces $\W$ to be $\la X, M[\name{G}_\delta],\w\ra$-hitting.
This completes our proof.
\end{proof}

\begin{lemma} \label{examples}
The Miller, Sacks, and Cohen  posets satisfy $(\dagger)$.
\end{lemma}
\begin{proof}
We shall present the proof only for  Miller forcing because it is exactly what is needed for the
proof of Theorem~\ref{main}, and because  the Sacks case is completely analogous, whereas
the Cohen one  is trivial.

Before we pass to the proof,
let us  recall the definition of  Miller forcing and fix our notation.
By a Miller tree we understand a subtree $T$ of $\w^{<\w}$ consisting of increasing finite sequences
such that the following conditions are satisfied:
\begin{itemize}
 \item Every $t\in T$ has an extension $s\in T$ which is splitting in $T$, i.e.,
there are more than one immediate successors of $s$ in $T$;
\item If $s$ is splitting in $T$, then it has infinitely many immediate successors in $T$.
\end{itemize}
  Miller forcing is the collection $\mathbb M$ of all  Miller trees ordered
by inclusion, i.e.,
smaller trees carry more information about the generic.
This poset was introduced in \cite{Mil84}.

For a Miller tree $T$ we  denote by $\Split(T)$ the set of all splitting nodes
of $T$, and for some $t\in\Split(T)$ we denote the size of $\{s\in\Split(T):s\subsetneq t\}$
by $\Lev(t,T)$. For a node $t$ in a Miller tree $T$ we denote by $T_t$
the set $\{s\in T:s$ is compatible with $t\}$. It is clear that $T_t$ is also a Miller
tree.
If $T_1\leq T_0 $ and each $t\in \Split(T_0)$ with $\Lev(t,T_0)\leq k$ belongs to $\Split(T_1)$, where $k\in\w$,
 then we  write
$T_1\leq_k T_0$. It is easy to check (and is well-known) that if $T_{n+1}\leq_n T_n$ for all $n\in\w$,
then $\bigcap_{n\in\w}T_n\in\mathbb M$.

We are now in a position to start the proof. Let $M$ and  $\{\phi_i:i\in\w\}$ be such as in the formulation of
$(\dagger)$. We can additionally assume that for each $\phi\in\{\phi_i:i\in\w\}$ there are infinitely many
$i$ such that $\phi=\phi_i$. Let $\{D_n:n\in\w\}$ be the set of all open dense subsets
of $\mathbb M$ which belong to $M$. Given $T_0\in M\cap\mathbb M$,
construct a sequence $\la T_n:n\in\w\ra\in\mathbb M^\w$  as follows:
Assume that  $T_n$ has been constructed such that $(T_n)_t\in M$  for every $t\in T_n$ with $\Lev(t,T_n)=n$.
Given such a $t\in T_n$  and $k\in\w$ such that
$t\vid k\in T_n$, find $R_{t,k}\leq\phi_n((T_n)_{t\vid k})$ such that $R_{t,k}\in D_n\cap M$.
Now set $T_{n+1}=\bigcup\{R_{t,k}:t\in T_k,\Lev(t,T_n)=n, t\vid k\in T_n\}$
and note that $T_{n+1}\leq_n T_n$ and $(T_{n+1})_r\in M$
for all $r\in T_{n+1}$ with $\Lev(r,T_{n+1})=n+1$.
This completes our construction. It is straightforward to check that  $T=\bigcap_{n\in\w}T_n$ is a $(M,\mathbb M)$-generic
condition forcing $\name{G}\cap\phi_n[M\cap\mathbb M]$ to be infinite for all $n$.
 \end{proof}

Finally we have all necessary ingredients to complete the proof of Theorem~\ref{main}.
Let $V$ be a model of GCH. By \cite[Theorem~3.2]{MilTsaZdo16} there exist $\gamma$-subspaces
$X,Y$ of $2^\w$ and a continuous map $\phi:X\times Y\to\w^\w$ such that
$\phi[X\times Y]$ is dominating, i.e., for every $f\in \w^\w$ there exists
$\la x,y\ra\in X\times Y$  such that $f\leq^*\phi\la x,y\ra$. (As usual, $f\leq^* g$ for $f,g\in\w^\w$
means that the set $\{n\in\w: f(n)>g(n)\}$ is finite. Whenever we speak about unbounded or dominating subsets of
$\w^\w$, we always mean with respect to $\leq^*$.) Let $\IP$ be
the iteration of $\mathbb M$ of length $\w_2$ with countable supports, and
$G$ be $\IP$-generic.
It is well known that $V\cap \w^\w$ is unbounded\footnote{Even more is true:
there exists an ultrafilter $\U\in V$ which remains a base for an ultrafilter in $V[G]$, namely all $P$-points are like that, see \cite{BlaShe89}.
It is easy to see that the set of enumerating functions of a base of an ultrafilter
cannot be bounded.} in $V[G]$,
and hence so is $\phi[X\times Y]$. By a result of Hurewicz \cite{Hur27}
(see also \cite[Theorem~4.3]{COC2}) this implies that $X\times Y$ is not Hurewicz in $V[G]$.
On the other hand, $X$ and $Y$ remain $\gamma$-spaces in $V[G]$ by a combination
of Lemmata~\ref{just_def},
\ref{dad_impl_gampres}, \ref{dag_preservation}, and \ref{examples}.
This completes our proof.  

\medskip

\noindent\textbf{Acknowledgments.}
The authors wish to express their  thanks to the referee whose suggestions have improved our exposition in
Lemma~\ref{dag_preservation}.

\end{document}